\documentclass[12pt, a4paper]{article}
\usepackage{amsmath, amssymb, amsthm}
\usepackage{graphics}
\usepackage{graphicx}
\usepackage{epsfig}
\usepackage{dblfloatfix}
\usepackage{float}
\usepackage{placeins}
\usepackage{flafter}
\usepackage[usenames]{color}
\usepackage{cite}

\usepackage{epstopdf}
\usepackage{url}
\usepackage{hyperref}
\date{}
\newtheorem{df}{Definition}[section]
\newtheorem{pro}{Proposition}[section]
\newtheorem{thm}{Theorem}[section]
\newtheorem{cor}{Corollary}[section]
\newtheorem{ex}{Example}[section]
\newtheorem{lem}{Lemma}[section]
\newtheorem{rem}{Remark}[section]

\makeatletter

\renewcommand\section{\@startsection {section}{1}{\z@}
{-30pt \@plus -1ex \@minus -.2ex} {2.3ex \@plus.2ex}
{\normalfont\normalsize\bfseries}}

\renewcommand\subsection{\@startsection{subsection}{2}{\z@}
{-3.25ex\@plus -1ex \@minus -.2ex} {1.5ex \@plus .2ex}
{\normalfont\normalsize\bfseries}}

\renewcommand{\@seccntformat}[1]{\csname the#1\endcsname. }

\makeatother

\begin{document}

\title{A note on lower nil $M$-Armendariz rings}
\author{\small{Sushma Singh and Om Prakash} \\ \small{Department of Mathematics} \\ \small{Indian Institute of Technology Patna, Bihta} \\ \small{Patna, INDIA- 801 106} \\  \small{sushmasingh@iitp.ac.in \& \;om@iitp.ac.in }}

\maketitle
\begin{abstract} In this article, we prove some results for lower nil $M$-Armendariz ring. Let $M$ be a strictly totally ordered monoid and $I$ be a semicommutative ideal of $R$. If $\frac{R}{I}$ is a lower nil $M$-Armendariz ring, then $R$ is lower nil $M$-Armendariz. Similarly, for above $M$, if $I$ is 2-primal with $N_{*}(R) \subseteq I$ and $R/I$ is $M$-Armendariz, then $R$ is a lower nil $M$-Armendariz ring. Further, we observe that if $M$ is a monoid and $N$ a u.p.-monoid where $R$ is a $2$-primal $M$-Armendariz ring, then $R[N]$ is a lower nil $M$-Armendariz ring.
\end{abstract}

\section{Introduction}
Throughout this article, $R$ denotes an associative ring with identity unless otherwise stated. For a ring $R$, $N_{*}(R)$, $N^{*}(R)$ and $N(R)$ denote the prime radical (lower nilradical), upper nilradical and the set of nilpotent elements of $R$ respectively. It is known that $N_{*}(R) \subseteq N^{*}(R) \subseteq N(R)$. Here, $R[x]$ denotes the polynomial ring with an indeterminate $x$ over $R$ and $C_{f(x)}$ for the set of all coefficients of $f(x) \in R[x]$. $M_{n}(R)$ and $T_{n}(R)$ represent the full matrix ring and upper triangular matrix ring of order $n$ over the ring $R$ respectively.\\
A ring $R$ is said to be Armendariz ring if whenever two polynomials, $f(x)g(x) = 0$ implies $ab = 0$ for each $a \in C_{f(x)}$, $b \in C_{g(x)}$. $R$ is reduced if it has no nonzero nilpotent element. In 1974, Armendariz himself proved that every reduced ring is satisfying above condition \cite{E}. Later, the term Armendariz ring coined by Rege and Chhawchharia \cite{MM} in 1997. Currently, several generalizations of Armendariz rings have been introduced.\\

A ring $R$ is said to be semicommutative if whenever $ab = 0$ implies $aRb = 0$ for $a, b \in R$ and it is 2-primal if $N_{*}(R) = N(R)$.  In 2001, Marks called a ring $R$ is $NI$ if $N^{*}(R) = N(R)$ \cite{GGG}.\\
 Let $M$ be a monoid and $e$ the identity element of $M$. $R[M]$ denotes the monoid ring of $M$ over ring $R$. In 2005, Liu \cite{Z}, defined $M$-Armendariz ring as if whenever elements $\alpha = a_{1}g_{1}+a_{2}g_{2}+\ldots+a_{n}g_{n}$, $\beta = b_{1}h_{1}+b_{2}h_{2}+\ldots+b_{m}h_{m} \in R[M]$ satisfy $\alpha \beta = 0$, then $a_{i}b_{j} = 0$ for each $i$ and $j$. Clearly, every ring is considered as $M$-Armendariz ring where $M = \{e\}$. Moreover, if $S$ is a semigroup with multiplication $st = 0$ for all $s, t \in S$ and $M = S^{1}$, the semigroup $S$ with identity, then ring is not an $M$-Armendariz. If $M = N \cup\{0\}$, then $R$ is $M$-Armendariz if and only if $R$ is Armendariz ring. Also, a monoid $M$ with $|M| \geq 2$ and $R$ is $M$-Armendariz, he proved that $R$ is a p.p.-ring if and only if $R[M]$ is a p.p.-ring.\\

In 2014, Alhevaz and Moussavi \cite{AA}, called a ring $R$ to be a nil $M$-Armendariz if  $\alpha = a_{1}g_{1}+a_{2}g_{2}+\ldots+a_{n}g_{n}$, $\beta = b_{1}h_{1}+b_{2}h_{2}+\ldots+b_{m}h_{m} \in R[M]$ such that $\alpha \beta \in N(R)[M]$, then $a_{i}b_{j} \in N(R)$ for each $i$ and $j$. Later, in 2016, Alhevaz and Hashemi \cite{A} introduced the concept of upper and lower nil $M$-Armendariz ring relative to a monoid. A ring $R$ is upper $(respectively,~ lower)$ nil $M$-Armendariz ring if whenever elements $\alpha = a_{1}g_{1}+a_{2}g_{2}+\ldots+a_{n}g_{n}$, $\beta = b_{1}h_{1}+b_{2}h_{2}+\ldots+b_{m}h_{m} \in R[M]$ such that $\alpha \beta \in N^{*}(R)[M]$ $( respectively~ N_{*}(R)[M] )$, then $a_{i}b_{j} \in N^{*}(R)$ $( respectively ~a_{i}b_{j} \in N_{*}(R))$ for each $i$ and $j$.

\section{Results on Lower nil M-Armendariz Ring}
Recall that a monoid $M$ is said to be u.p. monoid (unique product monoid) if for any two non empty finite subsets $A, B$ of $M$, there exists an element $g \in M$ uniquely presented in the form of $ab$ where $a \in A$ and $b \in B$. The class of unique product monoids is quite large and important. For example, the right or left ordered monoids, torsion free nilpotent groups, submonoids of a free group. unique product groups and u.p. monoids providing zero divisor problem for group ring. The ring theoretical property of unique product monoid has been established by many authors \cite{J, Z, JJ, D} in past few decades.
\begin{pro}\label{pro1} For a unique product monoid $M$, every $2$-primal ring is lower nil $M$-Armendariz ring.
\end{pro}
\begin{proof} Let $M$ be a unique product monoid. Let $\alpha = a_{1}g_{1}+a_{2}g_{2}+\ldots+a_{n}g_{n}$, $\beta = b_{1}h_{1}+b_{2}h_{2}+\ldots+b_{m}h_{m} \in R[M]$ such that $\alpha \beta \in N_{*}(R)[M]$. Then $\overline{\alpha}\overline{\beta} = \overline{0}$ in $R/N_{*}(R)[M]$. Since $R/N_{*}(R)$ is reduced, therefore, by Proposition $(1.1)$ of \cite{Z}, $R/N_{*}(R)$ is $M$-Armendariz ring. This implies $a_{i}b_{j} \in N_{*}(R)$ for each $i$ and $j$. Thus, $R$ is a lower nil $M$-Armendariz ring.
\end{proof}
\begin{cor}\label{cor1} For any unique product monoid $M$, semicommutative ring is lower nil $M$-Armendariz ring.
\end{cor}
A monoid $M$ equipped with an order $\leq$ is said to be an ordered monoid if for any $r_{1}, r_{2}, s \in M$, $r_{1}\leq r_{2}$ implies $r_{1}s \leq r_{2}s$ and $sr_{1}\leq sr_{2}$. Moreover, if $r_{1}< r_{2}$ implies $r_{1}s<r_{2}s$ and $sr_{1}<sr_{2}$, then $M$ is said to be strictly totally ordered monoid.\\
Since each strictly totally ordered monoid is a u.p. monoid, hence by Proposition \ref{pro1}, we have the following result.
\begin{cor} Let $M$ be a strictly totally ordered monoid. Then $2$-primal rings are lower nil $M$-Armendariz rings.
\end{cor}
The following example shows that the condition u.p. monoid in Proposition \ref{pro1} is not superfluous.
\begin{ex} Let $M = \{0, I, E_{11}, E_{12}, E_{21}, E_{22}\}$, where $E_{ij}$ is the unit matrix of $M_{2}(\mathbb{Z})$ for each $1 \leq i, j \leq 2$. Then $M$ is a monoid but it is not a u.p. monoid. Let $\alpha = 1.E_{22}$ and $\beta = 1.E_{11}-1.E_{12}$. Then $\alpha\beta = 0 \in N_{*}(\mathbb{Z})[M]$ but $1.1 \notin N_{*}(\mathbb{Z})$. Hence, $\mathbb{Z}$ is a not a lower nil $M$-Armendariz ring.
\end{ex}
\begin{pro} Let $M$ be a u.p. monoid, $R$ be a lower nil $M$-Armendariz ring and $S$, a subring of $R$.
\begin{itemize}
\item[$(1)$] If $N_{*}(S)\subseteq N_{*}(R)$, then $S$ is a lower nil $M$-Armendariz ring.
\item[$(2)$] If $R$ is an NI ring, then $S$ is a lower nil $M$-Armendariz ring.
\end{itemize}
\end{pro}
\begin{proof}
\begin{itemize}
\item[$(1)$] Let $\alpha \beta \in N_{*}(S)[M]$, where $\alpha, \beta \in R[M]$. Then by assumption, $\alpha \beta \in N_{*}(R)[M]$. Since $R$ is a lower nil $M$-Armendariz ring, therefore $a_{i}b_{j} \in N_{*}(R)$ for each $i, j$. Since $N_{*}(S)\subseteq N_{*}(R)$, so $N_{*}(S) = S\cap N_{*}(R)$ and hence $a_{i}b_{j} \in N_{*}(S)$ for each $i, j$.
\item[$(2)$] Since $R$ is NI, $N^{*}(R) = N(R)$. Also, for lower nil $M$-Armendariz ring, we have $N_{*}(R) = N^{*}(R)$. Therefore, $N^{*}(R) = N_{*}(R) = N(R)$. Thus, $S$ is a lower nil $M$-Armendariz ring.
\end{itemize}
\end{proof}

\begin{pro}\label{pro6} For a u.p. monoid $M$, every lower nil $M$-Armendariz ring is nil $M$-Armendariz ring.
\end{pro}
\begin{proof} We know that $R$ is a lower nil $M$-Armendariz ring if and only if $R/N_{*}(R)$ is an $M$-Armendariz ring. Then by using Lemma $(3.1(c))$ of \cite{A} and Theorem $(2.23)$ of \cite{EE}, we obtain the result.

\end{proof}
But converse is not true. In this regard, we have an example.
\begin{ex} Let $M$ be a u.p. monoid and we consider the ring given in Example 1.2 of \cite{S}. Let $S$ be a reduced ring, $n$ a positive integer and $R_{n} = U_{2^{n}}(S)$. Then each $R_{n}$ is an NI ring by [\cite{S}, Proposition 4.1(1)]. Define a map $\sigma : R_{n}\rightarrow R_{n+1}$ by $A\mapsto \left(
                                                                                                            \begin{array}{cc}
                                                                                                              A & 0 \\
                                                                                                              0 & A \\
                                                                                                            \end{array}
                                                                                                          \right)$,
then $R_{n}$ can be considered as subring of $R_{n+1}$ via $\sigma$ $(i.e., A = \sigma (A) ~for~ A \in R_{n} )$. Notice that $D = \{R_{n}, \sigma_{nm}\}$, with $\sigma_{nm} = \sigma^{m-n}$ whenever $n \leq m$, is a direct system over $I = \{1, 2, \ldots\}$. Set $R = \underrightarrow {Lim}~ R_{n}$ be the direct limit of $D$. Then $R = \bigcup_{n = 1}^{\infty}R_{n}$, and $R$ is NI by [\cite{S}, Proposition 1.1]. By Theorem $(2.2)$ of \cite{AA}, NI ring is nil $M$-Armendariz ring. By [\cite{YYY}, Theorem 2.2(1)], $R$ is a semiprime ring, hence $N_{*}(R) = 0$. Here $N^{*}(R) = \{m = (m_{ij})\in R ~|~ m_{ii} = 0 ~for ~all ~i\}$. Thus, by Lemma $(3.1(c))$ of \cite{A}, $R$ is not a lower nil $M$-Armendariz ring.
\end{ex}

\begin{pro} Let $M$ be a monoid with $|M|\geq 2$. Then the following condition are equivalent:
\begin{itemize}
\item[$(1)$] $R$ is a lower nil $M$-Armendariz ring.
\item[$(2)$] $T_{n}(R)$ is a lower nil $M$-Armendariz ring.
\end{itemize}
\end{pro}
\begin{proof} $\textbf{(1)}\Rightarrow \textbf{(2)}:$ It is easy to see that there exists an isomorphism between rings $T_{n}(R)[M]$ and $T_{n}(R[M])$, defined by\\

$\sum_{i=1}^{n} \left(
           \begin{array}{cccc}
             a_{11}^{i} & a_{12}^{i} & \ldots & a_{1n}^{i} \\
             0 & a_{22}^{i} & \ldots & a_{nn}^{i} \\
             \vdots & \vdots & \ddots & \vdots \\
             0 & 0 & \ldots & a_{nn}^{i} \\
           \end{array}
         \right)g_{i}\mapsto \left(
           \begin{array}{cccc}
             \sum_{i=1}^{n} a_{11}^{i}g_{i} & \sum_{i=1}^{n} a_{12}^{i}g_{i} & \ldots & \sum_{i=1}^{n}a_{1n}^{i}g_{i} \\
             0 & \sum_{i=1}^{n}a_{22}^{i}g_{i} & \ldots & \sum_{i=1}^{n}a_{nn}^{i}g_{i} \\
             \vdots & \vdots & \ddots & \vdots \\
             0 & 0 & \ldots & \sum_{i=1}^{n}a_{nn}^{i}g_{i} \\
           \end{array}
         \right).$\\

Let $\alpha = A_{1}g_{1}+A_{2}g_{2}+\ldots+A_{m}g_{m}$ and $\beta = B_{1}h_{1}+B_{2}h_{2}+\ldots+B_{n}h_{n} \in T_{n}(R)[M]$ such that $\alpha \beta \in N_{*}(T_{n}(R))[M]$, where\\

  $A_{i} = \left(
           \begin{array}{cccc}
             a_{11}^{i} & a_{12}^{i} & \ldots & a_{1n}^{i} \\
             0 & a_{22}^{i} & \ldots & a_{nn}^{i} \\
             \vdots & \vdots & \ddots & \vdots \\
             0 & 0 & \ldots & a_{nn}^{i} \\
           \end{array}
         \right)$~,~~~
         $B_{j} = \left(
           \begin{array}{cccc}
             b_{11}^{j} & b_{12}^{j} & \ldots & b_{1n}^{j} \\
             0 & b_{22}^{j} & \ldots & b_{nn}^{i} \\
             \vdots & \vdots & \ddots & \vdots \\
             0 & 0 & \ldots & b_{nn}^{j} \\
           \end{array}
         \right) \in T_{n}(R)$.\\

We have $N_{*}(T_{n}(R)) = \{m = (m_{ij}) \in T_{n}(R) ~|~ m_{ii} \in N_{*}(R)\}$. Therefore, from $\alpha \beta \in N_{*}(T_{n}(R))[M]$, we get $(\sum _{i=0}^{r}a_{pp}^{i}g_{i})(\sum_{j=0}^{s}b_{pp}^{j}h_{j}) \in N_{*}(R)[M]$, for $p = 1, 2\ldots n$. Since $R$ is a lower nil $M$-Armendariz ring, therefore $a_{pp}^{i}b_{pp}^{j} \in N_{*}(R)$, for each $p$ and all $i$, $j$ and hence $A_{i}B_{j} \in N_{*}(T_{n}(R))$ for each $i, j$. Thus, $T_{n}(R)$ is a lower nil M-Armendariz ring.\\

$\textbf{(2)}\Rightarrow\textbf{(1)}:$ Note that $R$ is isomorphic to $\left\{\left(
           \begin{array}{cccc}
             a & a & \ldots & a \\
             0 & a & \ldots & a \\
             \vdots & \vdots & \ddots & \vdots \\
             0 & 0 & \ldots & a \\
           \end{array}
         \right): a \in R \right\},$ subring of $T_{n}(R)$. Clearly, $R$ is a lower nil $M$-Armendariz ring.
\end{proof}
Now, below example makes it clear that full matrix ring $M_{2}(R)$ over a ring $R$ is not a lower nil $M$-Armendariz ring.
\begin{ex}
Let $M$ be a monoid with $|M|\geq 2$ and $R$ is a ring. Take $e \neq g \in M$ and let $\alpha = \left(
                                                                                               \begin{array}{cc}
                                                                                                 0 & 1 \\
                                                                                                 0 & 0 \\
                                                                                               \end{array}
                                                                                             \right)e -\left(
                                                                                                         \begin{array}{cc}
                                                                                                           1 & 0 \\
                                                                                                           0 & 0 \\
                                                                                                         \end{array}
                                                                                                       \right)g$ and
  $\beta = \left(
             \begin{array}{cc}
               1 & 1 \\
               0 & 0 \\
             \end{array}
           \right)e + \left(
                        \begin{array}{cc}
                          0 & 0 \\
                          1 & 1 \\
                        \end{array}
                      \right)g$. Then $\alpha\beta \in N_{*}(M_{2}(R))[M]$ but $\left(
                                                                                  \begin{array}{cc}
                                                                                    1 & 1 \\
                                                                                    0 & 0 \\
                                                                                  \end{array}
                                                                                \right) \notin N_{*}(M_{2}(R))$. Therefore, $M_{2}(R)$ is not a lower nil $M$-Armendariz ring.
\end{ex}
\begin{pro} Let $N$ be an ideal of a cancellative monoid $M$. If $R$ is a lower nil $N$-Armendariz ring, then $R$ is a lower nil $M$-Armendariz ring.
\end{pro}
\begin{proof} Suppose $\alpha = a_{1}g_{1}+a_{2}g_{2}+\ldots+a_{n}g_{n}$, $\beta = b_{1}h_{1}+b_{2}h_{2}+\ldots+b_{m}h_{m} \in R[M]$ such that $\alpha \beta \in N_{*}(R)[M]$. Take $k \in N$, then $kg_{1}, kg_{2},\ldots kg_{n}, h_{1}k, h_{2}k, \ldots h_{n}k \in N$ and $kg_{i} \neq kg_{j}$, $h_{i}k \neq h_{j}k$ when $i \neq j$. Now,
\begin{equation}
(\sum_{i = 1}^{n}a_{i}kg_{i})(\sum_{j = 1}^{m}b_{j}h_{j}k) \in N_{*}(R)[N].
\end{equation}
Since $R$ is a lower nil $N$-Armendariz ring, therefore $a_{i}b_{j} \in N_{*}(R)$ for each $i, j$. Thus, $R$ is a lower nil $M$-Armendariz ring.
\end{proof}
\begin{df} Let $R$ be a ring and $S^{-1}R = \{u^{-1}a ~|~ u \in S, a \in R\}$ with $S$ a multiplicative closed subset of $R$ consisting of central regular elements. Then $S^{-1}R$ is a ring.
\end{df}
\begin{pro} Let $M$ be a monoid and $S$ be a multiplicative monoid in a ring $R$ consisting of central regular elements. Then $R$ is a lower nil $M$-Armendariz ring if and only if so is $S^{-1}R$.
\end{pro}
\begin{proof} Let $R$ be a lower nil $M$-Armendariz ring and $\alpha \beta \in N_{*}(S^{-1}R)[M]$ where $\alpha = p_{1}g_{1}+p_{2}g_{2}+\ldots+p_{n}g_{n}$ and $\beta = q_{1}h_{1}+q_{2}h_{2}+\ldots+q_{m}h_{m} \in (S^{-1}R)[M]$. Here, we consider $p_{i} = a_{i}u^{-1}$ and $q_{j} = b_{j}v^{-1}$, where $a_{i}, b_{j} \in R$ for each $i, j$ and $u, v \in S$. From $\alpha \beta \in N_{*}(S^{-1}R)[M]$, we have $(a_{1}g_{1}+a_{2}g_{2}+\ldots+a_{n}g_{n})(b_{1}h_{1}+b_{2}h_{2}+\ldots+b_{m}h_{m}) \in N_{*}(R)[M]$. Since $R$ is the lower nil $M$-Armendariz ring, therefore $a_{i}b_{j} \in N_{*}(R)$. Since $N_{*}(S^{-1}(R)) = S^{-1}N_{*}(R)$, so $p_{i}q_{j} = a_{i}u^{-1}b_{j}v^{-1} = a_{i}b_{j}(uv)^{-1} \in N_{*}(S^{-1}(R))$. Thus, $S^{-1}R$ is a lower nil $M$-Armendariz ring.\\
Converse is obvious, since $N_{*}(R)\subseteq N_{*}(S^{-1}R)$ and $R$ is subring of $S^{-1}R$. Therefore, $R$ is a lower nil $M$-Armendariz ring.
\end{proof}
\begin{lem}\label{lem2} Let $M$ be a monoid with a nontrivial element of finite order and $R$ a ring such that $0 \neq 1$. Then $R$ is not a lower nil $M$-Armendariz ring.
\end{lem}
\begin{proof} Suppose $e \neq g \in M$ has order $n$ and consider $\alpha = 1e+1g+\ldots+1g^{n-1}$ and $\beta = 1e+(-1)g$. Then $\alpha \beta = 0$ i.e. $\alpha \beta \in N_{*}(R)[M]$ but $1 \notin N_{*}(R)$. Hence, $R$ is not a lower nil $M$-Armendariz ring.
\end{proof}

\begin{lem}\label{lem3} Let $N$ be a submonoid of the monoid $M$. If $R$ is a lower nil $M$-Armendariz ring, then $R$ is a lower nil $N$-Armendariz ring.
\end{lem}
\begin{df}If $T(G)$ contain elements of finite order in an abelian group $G$. Then $T(G)$ is a fully invariant subgroup of $G$. Group $G$ is said to be torsion-free group if $T(G) = \{e\}$.
\end{df}
\begin{thm} Let $G$ be a finitely generated abelian group. Then the following conditions on $G$ are equivalent:
\begin{itemize}
\item[$(1)$] $G$ is torsion-free.
\item[$(2)$] There exists a ring $R$ with $|R| \geq 2$ such that $R$ is a lower nil $G$-Armendariz ring.
\end{itemize}
\end{thm}
\begin{proof} $\textbf{(1)} \Rightarrow \textbf{(2):}$ If $G$ is a finitely generated abelian group with $T(G) = \{e\}$. Then \\
$G \cong \mathbb{Z}\times \mathbb{Z}\times\ldots \times \mathbb{Z}$. By [\cite{Z}, Lemma 1.13], $G$ is u.p.-monoid. Also, by corollary \ref{cor1}, $R$ is a lower nil $G$-Armendariz, for any semicommutative ring $R$ with $|R| \geq 2$,.\\
$\textbf{(2)}\Rightarrow \textbf{(1):}$ Let $g \in T(G)$ and $g \neq e$. Then $H = ~ <g>$ is a cyclic group of finite order. If a ring $R \neq 0$ is lower nil $G$-Armendariz ring, then $R$ is lower nil $H$-Armendariz ring by Lemma \ref{lem3}, which contradicts Lemma \ref{lem2}. Thus, every ring $R \neq 0$ is not a lower nil $G$-Armendariz ring. Hence, $g = e$ and $T(G) = {e}$.
\end{proof}

\begin{pro} Let $M$ be a monoid with $M\geq 2$. Then the following conditions are equivalent:
\begin{itemize}
\item[$(1)$] $R$ is a lower nil $M$-Armendariz ring.
\item[$(2)$]
$H_{3}(R)  = \left\{\left(
               \begin{array}{ccc}
                 a_{11} & 0 & 0 \\
                 a_{21} & a_{22} & a_{23} \\
                 0 & 0 & a_{33} \\
               \end{array}
             \right): a_{ij} \in R \right\}$ is a lower nil $M$-Armendariz ring.
\end{itemize}
\end{pro}
\begin{proof} First, we claim that
\begin{equation*}
N_{*}(H_{3}(R)) = \left(
                    \begin{array}{ccc}
                      N_{*}(R) & 0 & 0 \\
                      R & N_{*}(R) & R \\
                      0 & 0 & N_{*}(R) \\
                    \end{array}
                  \right).
\end{equation*}
If we consider $\left(
             \begin{array}{ccc}
               a_{11} & 0 & 0 \\
               a_{21} & a_{22} & a_{23} \\
               0 & 0 & a_{33} \\
             \end{array}
           \right) \in \left(
                    \begin{array}{ccc}
                      N_{*}(R) & 0 & 0 \\
                      R & N_{*}(R) & R \\
                      0 & 0 & N_{*}(R) \\
                    \end{array}
                  \right)$,\\

then $a_{ii} \in N_{*}(R)$ for each $1 \leq i \leq 3$. So, $Ra_{ii}R \subseteq N(R)$ for each $i$, $1 \leq i \leq 3$. Then there exists a positive integer $n$ corresponding to nilpotency of  all elements of $Ra_{ii}R$ such that $\left(
             \begin{array}{ccc}
               a_{11} & 0 & 0 \\
               a_{21} & a_{22} & a_{23} \\
               0 & 0 & a_{33} \\
             \end{array}
           \right)^{2n} = 0  \in N_{*}(H_{3}(R))$. Hence, $\left(
                    \begin{array}{ccc}
                      N_{*}(R) & 0 & 0 \\
                      R & N_{*}(R) & R \\
                      0 & 0 & N_{*}(R) \\
                    \end{array}
                  \right) \subseteq N_{*}(H_{3}(R))$.\\

Conversely, $\left(
             \begin{array}{ccc}
               a_{11} & 0 & 0 \\
               a_{21} & a_{22} & a_{23} \\
               0 & 0 & a_{33} \\
             \end{array}
           \right) = C \in N_{*}(H_{3}(R))$. Then $(H_{3}(R))C(H_{3}(R)) \subseteq N(H_{3}(R))$. This implies $Ra_{ii}R \subseteq N(R)$ and hence $a_{ii} \in N_{*}(R)$ for each $1 \leq i \leq 3$.
Therefore
\begin{equation*}
N_{*}(H_{3}(R)) \subseteq \left(
                    \begin{array}{ccc}
                      N_{*}(R) & 0 & 0 \\
                      R & N_{*}(R) & R \\
                      0 & 0 & N_{*}(R) \\
                    \end{array}
                  \right).
\end{equation*}
Next, we prove  $H_{3}(R)$ is a lower nil $M$-Armendariz ring. Since $\Phi : H_{3}(R)[M]\mapsto H_{3}(R[M])$ defined by
\begin{equation*}
\sum_{i=1}^{n}\left(
                \begin{array}{ccc}
                  a_{11}^{i} & 0 & 0 \\
                  a_{21}^{i} & a_{22}^{i} & a_{23}^{i} \\
                  0 & 0 & a_{33}^{i} \\
                \end{array}
              \right)g_{i} \mapsto \left(
                \begin{array}{ccc}
                  \sum_{i=1}^{n}a_{11}^{i}g_{i} & 0 & 0 \\
                  \sum_{i=1}^{n}a_{21}^{i}g_{i} & \sum_{i=1}^{n}a_{22}^{i}g_{i} & \sum_{i=1}^{n}a_{23}^{i}g_{i} \\
                  0 & 0 & \sum_{i=1}^{n}a_{33}^{i}g_{i} \\
                \end{array}
              \right)
\end{equation*}
is an isomorphism. Let $\alpha = A_{1}g_{1}+A_{2}g_{2}+\cdots+A_{n}g_{n}$, $\beta = B_{1}h_{1}+B_{2}h_{2}+\cdots+B_{m}h_{m} \in H_{3}(R)[M]$ such that $\alpha \beta \in N_{*}(H_{3}(R))[M]$, where $A_{i}, B_{j} \in H_{3}(R)$. Also, let $\alpha_{p} = \sum_{i=1}^{n}a_{pp}^{i}g_{i}$, $\beta_{p} = \sum_{j=1}^{m}b_{pp}^{j}h_{j}\in R[M]$, for $p = 1, 2, 3.$  Then by above isomorphism, it is easy to see that $\alpha \beta \in N_{*}(H_{3}(R)[M])$. So $\alpha_{p}\beta_{p} \in N_{*}(R)[M]$ and this implies $a_{pp}^{i}b_{pp}^{j} \in N_{*}(R)$ for each $i, j, p$. Therefore, $A_{i}B_{j} \in N_{*}(H_{3}(R))$. Thus, $H_{3}(R)$ is a lower nil $M$-Armendariz ring.
\end{proof}
\begin{pro}
\begin{itemize}
\item[$(1)$] Direct sum of lower nil $M$-Armendariz rings is a lower nil $M$-Armendariz ring.
\item[$(2)$] Direct product of lower nil $M$-Armendariz rings is a lower nil $M$-Armendariz ring.
\end{itemize}
\end{pro}
\begin{proof}
\begin{itemize}
\item[$(1)$] It is well known that $N_{*}(\prod R_{i}) = \prod(N_{*}(R_{i}))$. Let $\alpha = \sum_{i=0}^{m}a_{i_\delta}g_{i} \in R[M]$ and $\beta = \sum_{j=0}^{n}b_{j_\delta}h_{j} \in R[M]$ such that $\alpha \beta \in N_{*}(R)[M]$. Then $\alpha = (\alpha_{\delta})$ and $\beta = (\beta_{\delta}) \in R[M]$, where $\alpha_{\delta} = \sum_{i=0}^{m}a_{i_\delta}g_{i}$ and $\beta_{\delta} = \sum_{j=0}^{n}b_{j_\delta}g_{j}$. So, $\alpha_{\delta}\beta_{\delta} \in N_{*}(R_{\delta})$ where $\delta \in I$. Since $R_{\delta}$ is the lower nil $M$-Armendariz ring, therefore $a_{i_\delta}b_{j_\delta} \in N_{*}(R_{\delta})$. Thus, $R$ is a lower nil $M$-Armendariz ring.\\
\item[$(2)$]It is easily seen that $N_{*}(\oplus R_{i}) = \oplus(N_{*}(R_{i}))$. As above, direct sum of lower nil $M$-Armendariz ring is a lower nil $M$-Armendariz ring.
\end{itemize}
\end{proof}
\begin{pro}\label{pro2}
For a monoid $M$, if $R$ is a semicommutative lower nil $M$-Armendariz ring, then $N_{*}(R)[M] = N_{*}(R[M])$.
\end{pro}
\begin{proof} Let $\alpha = a_{1}g_{1}+a_{2}g_{2}+\cdots+a_{m}g_{m} \in N_{*}(R)[M]$, where $a_{i} \in N_{*}(R)$ for each $1\leq i \leq n$. Now, $Ra_{i}R \subseteq N(R)$ for each $1\leq i \leq n$. Then there exists $k_{i}$, a positive integer, such that $(Ra_{i}R)^{k_{i}} = 0$ for each fixed $a_{i}$, where $1\leq i \leq n$. Let $k > max \{k_{i}: 1\leq i \leq n\}$. Let $\beta_{1} = b_{1}h_{1}+b_{2}h_{2}+\cdots+b_{l}h_{l}$ and $\beta_{2} = c_{1}h_{1}^{'}+c_{2}h_{2}^{'}+\cdots+c_{n}h_{n}^{'} \in R[M]$. Then $\beta_{1}\alpha\beta_{2} = (b_{1}a_{1}c_{1}h_{1}g_{1}h_{1}^{'}+b_{1}a_{1}c_{2}h_{1}g_{2}h_{2}^{'}+\cdots+b_{1}a_{1}c_{n}h_{1}g_{1}h_{n}^{'}) + (b_{1}a_{2}c_{1}h_{1}g_{2}h_{1}^{'}+\cdots+b_{1}a_{2}c_{n}h_{1}g_{2}h_{n}^{'}) +\cdots+\\
(b_{1}a_{m}c_{1}h_{1}g_{m}h_{1}^{'}+\cdots+b_{1}a_{m}c_{n}h_{1}g_{m}h_{n}^{'}) + (b_{2}a_{1}c_{1}h_{2}g_{1}h_{1}^{'}+\cdots+b_{2}a_{m}c_{n}h_{2}g_{1}h_{n}^{'}) +\cdots+ (b_{2}a_{m}c_{1}h_{2}g_{m}h_{1}^{'}+\cdots+b_{2}a_{m}c_{n}h_{2}g_{m}h_{n}^{'})+\cdots+(b_{l}a_{1}c_{1}h_{l}g_{1}h_{1}^{'}+\cdots+b_{l}a_{1}c_{n}h_{l}g_{1}h_{n}^{'})+\cdots+(b_{l}a_{m}c_{1}h_{l}g_{m}h_{1}^{'}+\cdots+b_{l}a_{m}c_{n}h_{l}g_{m}h_{n}^{'})$. For brevity of notation, let $\beta_{1}\alpha\beta_{2} = A_{1}t_{1}+A_{2}t_{2}+\cdots+A_{lmn}t_{lmn}$. Then $(\beta_{1}\alpha\beta_{2})^{klmn} = (A_{1}t_{1}+A_{2}t_{2}+\cdots+A_{lmn}t_{lmn})^{klmn} = \sum_{u}(\sum_{t_{s_{1}}t_{s_{2}\ldots t_{s_{klmn}}}= u}{A_{s_{1}}A_{s_{2}}\ldots A_{s_{klmn}}})u$. Here, in $(\beta_{1}\alpha\beta_{2})^{klmn}$, each term of $A_{s_{1}}A_{s_{2}}\ldots A_{s_{klmn}}$, there exist some $A_{j}$ for $1 \leq j \leq lmn$, at least $k$ times. Therefore, we can replace $A_{s_{1}}A_{s_{2}}\ldots A_{s_{klmn}}$ by $B_{1}A_{j}^{v_{1}}B_{2}A_{j}^{v_{2}}\ldots B_{w}A_{j}^{v_{p}}$, where $v_{1}+v_{2}+\cdots+v_{p} > k$ and $B_{q}$ is a product of some elements from the set $\{A_{1}, A_{2}, \ldots, A_{lmn}\}$. Since $A^{v_{1}+v_{2}+\cdots+v_{p}} = 0$ and $R$ is a semicommutative ring, so $B_{1}A_{j}^{v_{1}}B_{2}A_{j}^{v_{2}}\ldots B_{w}A_{j}^{v_{p}} = 0$. Therefore, $A_{s_{1}}A_{s_{2}}\ldots A_{s_{klmn}} = 0$ and hence $(\beta_{1}\alpha\beta_{2})^{klmn} = 0$. Thus, $N_{*}(R)[M]\subseteq N_{*}(R[M])$.\\
Conversely, let $\alpha = a_{1}g_{1}+a_{2}g_{2}+\cdots+a_{m}g_{m} \in N_{*}(R[M])$. Then for each $\beta_{1}, \beta_{2} \in R[M]$, $(\beta_{1}\alpha\beta_{2})\subseteq N(R[M])$. In the same way, we get $(\gamma_{1}\alpha\gamma_{2})^{s} = 0$, for each $\gamma_{1}, \gamma_{2} \in R[M]$, since $\alpha \in N_{*}(R[M])$. As above, if $(\beta_{1}\alpha\beta_{2})^{n} = (A_{1}t_{1}+A_{2}t_{2}+\cdots+A_{lmn}t_{lmn})^{n} = 0$. This implies $(\beta_{1}\alpha\beta_{2})^{n} \in N_{*}(R)[M]$, therefore $A_{i}^{n} \in N_{*}(R)$, since $R$ is lower nil $M$-Armendariz ring. By the above expression, $(A_{1})^{n} = (b_{1}a_{1}c_{1})^{n} \in N_{*}(R)\subseteq N(R)$. This implies that $b_{1}a_{1}c_{1} \in N(R)$ for each $b_{1}, c_{1} \in R$, thus $a_{1} \in N_{*}(R)$.  Similarly, $a_{i} \in N_{*}(R)$ for each $1\leq i \leq m$. Therefore, $N_{*}(R[M])\subseteq N_{*}(R)[M]$. Hence, $N_{*}(R[M]) = N_{*}(R)[M]$.

\end{proof}

\begin{pro}\label{pro3} Let $M$ be a monoid and $N$ a u.p. monoid. If $R$ is a semicommutative lower nil $M$-Armendariz ring, then $R[M]$ is a lower nil $N$-Armendariz ring.
\end{pro}
\begin{proof} Here, every lower nil $M$-Armendariz ring is a nil $M$-Armendariz ring. Moreover, by Proposition 2.12 of \cite{AA}, $N(R)[M] = N(R[M])$ and by Proposition \ref{pro2}, $N_{*}(R)[M] = N_{*}(R[M])$. Again, by semicommutative of $R$, we get $N(R[M]) = N(R)[M] = N_{*}(R)[M] = N_{*}(R[M])$. Hence, $R[M]$ is a $2$-primal ring. Finally, by Proposition \ref{pro1}, $R[M]$ is a lower nil $N$-Armendariz ring.
\end{proof}
\begin{thm} Let $M$ be a monoid and $N$ be a u.p. monoid. If $R$ is a semicommutative and lower nil $M$-Armendariz ring, then $R[N]$ is a lower nil $M$-Armendariz ring.
\end{thm}
\begin{proof} First note that $\Phi : R[N][M] \rightarrow R[M][N]$ defined by $\phi(\sum_{i}(\sum_{j}a_{i_{j}}n_{j})m_{i}) = \sum_{j}(\sum_{i}a_{i_{j}}m_{i})\\n_{j}$ is a ring isomorphism. Now, suppose $(\sum_{i}\alpha_{i}m_{i})(\sum_{j}\beta_{j}m^{'}_{j}) \in N_{*}(R[N])[M]$, where $m_{i}$, $m^{'}_{j} \in M$ and $\alpha_{i}, \beta_{j} \in R[N]$ for each $i, j$. In order to prove $\alpha_{i}\beta_{j} \in N_{*}(R[N])$ for each $i, j$, let $\alpha_{i} = \sum_{r}a_{i_{r}}n_{r}$ and $\beta_{j} = \sum_{s}a_{j_{s}}n^{'}_{s} \in R[N]$. From Proposition \ref{pro2}, we have $N_{*}(R)[M] = N_{*}(R[M])$. Further, $R$ is a semicommutative ring and $N$ a u.p. monoid, so by Corollary \ref{cor1}, $R$ is a lower nil $N$-Armendariz ring. Again, by Theorem 3 of \cite{J}, we have $N_{*}(R)[N] = N_{*}(R[N])$. Hence, $(\sum_{r}(\sum_{i}a_{i_{r}}m_{i})n_{r})(\sum_{s}(\sum_{j}a_{j_{s}}m^{'}_{j})n^{'}_{s}) \in N_{*}(R[M])[N]$. Therefore, by Proposition \ref{pro2}, $(\sum_{i}a_{i_{r}}m_{i})(\sum_{j}a_{j_{s}}m^{'}_{j}) \in N_{*}(R[M]) = N_{*}(R)[M]$, for each $r$ and $s$. Since $R$ is a lower nil $M$-Armendariz ring, therefore $a_{i_{r}}b_{j_{s}} \in N_{*}(R)$ for each $i, j, r, s$, this implies $\alpha_{i}\beta _{j} \in N_{*}(R)[N] = N_{*}(R[N])$.  Thus, $R[N]$ is a lower nil $M$-Armendariz ring.
\end{proof}

A ring $R$ is said to be a Dedekind finite $(von~~Neumann-finite)$ if $ab = 1$ implies $ba = 1$ for each $a, b \in R$.
\begin{pro} Let $M$ be a cyclic group of order $n\geq 2$. Then each lower nil $M$-Armendariz ring is a Dedekind finite.
\end{pro}
\begin{proof} Let $R$ be a lower nil $M$-Armendariz ring and assume on the contrary $R$ is not a Dedekind finite. Then by Proposition $(5.5)$ of \cite{K}, $R$ contains an infinite set of matrix units, say $\{E_{11}, E_{12}, E_{13}, \ldots,E_{21}, E_{22}, E_{23},\ldots\}$. Consider the elements $\alpha = E_{11}e + E_{12}g$ and $\beta = (E_{22}-E_{11})e + (E_{11}-E_{12})g$ of $R[M]$. Then $\alpha\beta = 0$ but $(E_{11}-E_{12})$ is not a strongly nilpotent, which contradicts the assumption. Hence, $R$ is a Dedekind finite.
\end{proof}
\begin{thm}\label{thm2} Let $M$ be a monoid and $N$ a u. p. monoid. If $R$ is a semicommutative lower nil $M$-Armendariz ring, then $R$ is a lower nil $M\times N$-Armendariz ring.
\end{thm}
\begin{proof} By Theorem (2.3) of \cite{Z}, we have $R[M\times N] \cong R[M][N]$ and by Proposition \ref{pro2}, $N_{*}(R)[M] = N_{*}(R[M])$. Rest part of the proof follows Proposition \ref{pro3}.
\end{proof}
Let $M_{i}$, $i \in I$, be monoids for index set $I$. Denote $\bigsqcup_{i \in I} M_{i} = \{(h_{i})_{i \in I} \mid$ there exist only finite $i's$ such that $h_{i} \neq e_{i},$ the identity of $M_{i}\}$. Then $\bigsqcup_{i \in I}M_{i}$ is a monoid under the binary operation $(h_{i})_{i \in I}(h_{i}^{'})_{i \in I} = (h_{i}h_{i}^{'})_{i \in I}$.
\begin{cor} Let $M_{i}, i \in I$ be u.p. monoids and $R$ a semicommutative ring. If $R$ is a lower nil $M_{i_{0}}$-Armendariz ring for some $i_{0} \in I$, then $R$ is a lower nil $(\bigsqcup_{i \in I}M_{i})$-Armendariz ring.
\end{cor}
\begin{proof} Let $\alpha = \sum_{i}a_{i}g_{i}$, $\beta = \sum_{j}b_{j}h_{j} \in R[\bigsqcup_{i \in I}M_{i}]$ such that $\alpha \beta \in N_{*}(R)[\bigsqcup_{i \in I}M_{i}]$. Then $\alpha, \beta \in R[M_{i_{0}}\times M_{1}\times M_{2}\times\cdots \times M_{n}]$ for some finite subset $\{M_{1}, M_{2}, \ldots, M_{n}\}\subseteq \{M_{i} : i \in I\}$. From Theorem \ref{thm2} and by applying induction, the ring $R$ is a lower nil $(M_{i_{0}}\times M_{1} \times M_{2} \times \cdots \times M_{n})$-Armendariz ring, therefore $a_{i}b_{j} \in N_{*}(R)$. Hence, $R$ is a lower nil $\bigsqcup_{i \in I}M_{i}$-Armendariz ring.
\end{proof}
\begin{thm} The classes of lower nil $M$-Armendariz rings are closed under direct limit.
\end{thm}
\begin{proof} Let $D = \{R_{i}, \alpha_{ij}\}$ be a direct system of lower nil $M$-Armendariz rings $R_{i}$ for $i \in I$ and ring homomorphisms $\alpha_{ij} : R_{i}\rightarrow R_{j}$ for each $i \leq j$ satisfying $\alpha_{ij}(1) = 1$, where $I$ is a directed partially ordered set. Let $R = \underrightarrow{lim}R_{i}$  be the direct limit of $D$ with $l_{i} : R_{i}\rightarrow R$ and $l_{i}\alpha_{ij} = l_{i}$. We will prove $R$ is lower nil $M$-Armendariz ring. If we take $a, b \in R$, then $a = l_{i}(a_{i})$, $b = l_{j}(b_{j})$ for some $i, j \in I$ and there is $s \in I$ such that $i\leq s$, $j \leq n$. Define
\begin{eqnarray*}
a+b = l_{s}(\alpha_{is}(a_{i})+\alpha_{js}(b_{j})) ~~~and~~~ ab = l_{s}(\alpha_{is}(a_{i})\alpha_{js}(b_{j}))
\end{eqnarray*}
where $\alpha_{is}(a_{i})$, $\alpha_{js}(b_{j}) \in R_{s}$. Then $R$ forms a ring with $l_{i}(0) = 0$ and $l_{i}(1) = 1$. Now, let $\alpha \beta \in N_{*}(R)[M]$ for $\alpha = \sum_{p=1}^{m}a_{p}g_{p}$ and $\beta = \sum_{q=1}^{n}b_{q}h_{q} \in R[M]$. There are $i_{p}, j_{q}, s \in I$ such that $a_{p} = l_{i_{p}}(a_{i_{p}})$, $b_{q} = l_{i_{q}}(b_{j_{q}})$, $i_{p}\leq s$, $j_{q}\leq s$ and hence $\alpha \beta \in N_{*}(R_{s})[M]$. Since $R_{s}$ is a lower nil $M$-Armendariz, $l_{s}(\alpha_{i_{p}s}(a_{i})\alpha_{j_{q}s}(b_{j})) \in N_{*}(R_{s})$ and therefore $a_{p}b_{q} \in N_{*}(R)$. Thus, $R$ is a lower nil $M$-Armendariz ring.
\end{proof}

\begin{thm}
Let $M$ be a strictly totally ordered monoid and $I$ be an ideal of the ring $R$. If $I$ is semicommutative ideal and $R/I$ is a lower nil $M$-Armendariz ring, then $R$ is a lower nil $M$-Armendariz.
\end{thm}
\begin{proof}
Let $\alpha = \sum_{i=1}^{m}a_{i}g_{i}$ and $\beta = \sum_{j=1}^{n}b_{j}h_{j}$ in $R[M]$ with $\alpha \beta \in N_{*}(R)[M]$ where $g_{1} < g_{2} < g_{3}\cdots g_{m},$ $h_{1} < h_{2} < h_{3}< \cdots h_{n}$. Then $\sum_{g_{i}h_{j} = w}a_{i}b_{j} \in N_{*}(R)$. Now, we use transfinite induction on the strictly totally ordered set $(M, \leq)$ to prove $a_{i}b_{j} \in N_{*}(R)$, for each $i,j$. Clearly, $\overline{\alpha}\overline{\beta} \in N_{*}(R/I)[M]$. Since $R/I$ is a lower nil $M$-Armendariz ring, so there exists a positive integer $s$ such that $(r_{1}a_{i}b_{j}r_{2})^{s} \in N(R)$ for each $i, j$ and $r_{1}, r_{2} \in R$. Also, $g_{1}h_{1} < g_{i}h_{i}$ if $i \neq 1$ or $j \neq 1$. This implies that $a_{1}b_{1} \in N_{*}(R)$. Now, suppose $a_{i}b_{j} \in N_{*}(R)$, for any $g_{i}$ and $h_{j}$ with $g_{i}h_{j} < w$. In order to prove $a_{i}b_{j} \in N_{*}(R)$, for each $i$, $j$ and $g_{i}h_{j} = w$, consider $X = \{(g_{i}, h_{j}) | g_{i}h_{j} = w\}$. Then $X$ is a finite set. We write $X$ as $\{(g_{i_{t}}, h_{j_{t}})| t = 1, 2, \ldots, k\}$ such that $g_{i_{1}} < g_{i_{2}} < \cdots < g_{i_{k}}$. Since $M$ is a cancellative monoid, $g_{i_{1}} = g_{i_{2}}$ and $g_{i_{1}}h_{j_{1}} = g_{i_{2}}h_{j_{2}} = w$ imply that $h_{j_{1}} = h_{j_{2}}$. Also, since $(M, \leq)$ is a strictly totally ordered monoid, $g_{i_{1}} < g_{i_{2}}$ and $g_{i_{1}}h_{j_{1}} = g_{i_{2}}h_{j_{2}} = w$ imply $h_{j_{2}} < h_{j_{1}}$. So, we have $h_{j_{k}} < \cdots < h_{j_{2}} < h_{j_{1}}$. Now, we have
\begin{eqnarray*}
\sum_{g_{i}h_{j} = w}a_{i}b_{j} = \sum_{(g_{i}, h_{j}) \in X}a_{i}b_{j} = \sum_{t=1}^{k}a_{i_{t}}b_{j_{t}} \in N_{*}(R)\subseteq N(R).
\end{eqnarray*}
\begin{eqnarray}
\sum_{g_{i}h_{j} = w}r_{1}a_{i}b_{j}r_{2} = \sum_{(g_{i}, h_{j}) \in X}r_{1}a_{i}b_{j}r_{2} = \sum_{t=1}^{k}r_{1}a_{i_{t}}b_{j_{t}}r_{2} \in N(R) ~for~ any~ r_{1}, r_{2} \in R.
\end{eqnarray}
Note that, $g_{i_{1}}h_{j_{t}} < g_{i_{t}}h_{j_{t}} = w$ for any $t \geq 2$ and by induction hypothesis, we have $a_{i_{1}}b_{j_{2}} \in N_{*}(R)$. Consider $(a_{i_{1}}b_{j_{2}})^{p} = 0$ for some integer $p$. Then $(b_{j_{2}}a_{i_{1}})^{p+1} = 0$. Therefore,
\begin{eqnarray*}
((r_{1}a_{i_{2}}b_{j_{2}}r_{2})(r_{1}a_{i_{1}}b_{j_{1}}r_{2})^{s+1}r_{1}a_{i_{2}})(b_{j_{2}}a_{i_{1}})^{p+1}(b_{j_{2}}a_{i_{1}}(r_{1}a_{i_{1}}b_{j_{1}}r_{2})^{s+1}) = 0.
\end{eqnarray*}
\begin{eqnarray*}
((r_{1}a_{i_{2}}b_{j_{2}}r_{2})(r_{1}a_{i_{1}}b_{j_{1}}r_{2})^{s+1}r_{1}a_{i_{2}})(b_{j_{2}}a_{i_{1}})(b_{j_{2}}a_{i_{1}})^{p}(b_{j_{2}}a_{i_{1}}(r_{1}a_{i_{1}}b_{j_{1}}r_{2})^{s+1}) = 0.
\end{eqnarray*}
\begin{eqnarray*}
((r_{1}a_{i_{2}}b_{j_{2}}r_{2})(r_{1}a_{i_{1}}b_{j_{1}}r_{2})^{s+1}r_{1}a_{i_{2}})(b_{j_{2}})(r_{1}(r_{1}a_{i_{1}}b_{j_{1}}r_{2})^{s+1}r_{2})(a_{i_{1}})(b_{j_{1}}r_{2}(r_{1})a_{i_{2}}b_{j_{2}}r_{2})\\(r_{1}a_{i_{1}}b_{j_{1}}r_{2})^{s+1}r_{1}a_{i_{2}})(b_{j_{2}}a_{i_{1}})^{p}(b_{j_{2}}a_{i_{1}}(r_{1}a_{i_{1}}b_{j_{1}}r_{2})^{s+1}) = 0.
\end{eqnarray*}
This implies
\begin{eqnarray*}
[(r_{1}a_{i_{2}}b_{j_{2}}r_{2})(r_{1}a_{i_{1}}b_{j_{1}}r_{2})^{s+1}]^{2}[(r_{1}a_{i_{2}}b_{j_{2}}r_{2})(r_{1}a_{i_{1}}b_{j_{1}}r_{2})^{s+1}r_{1}a_{i_{2}}](b_{j_{2}})(r_{1}(r_{1}a_{i_{1}}b_{j_{1}}r_{2})^{s+1}r_{2})(a_{i_{1}})\\(b_{j_{1}}r_{2}(r_{1})a_{i_{2}}b_{j_{2}}r_{2})(r_{1}a_{i_{1}}b_{j_{1}}r_{2})^{s+1}r_{1}a_{i_{2}})(b_{j_{2}}a_{i_{1}})^{p-1}(b_{j_{2}}a_{i_{1}}(r_{1}a_{i_{1}}b_{j_{1}}r_{2})^{s+1}) = 0.
\end{eqnarray*}
Continuing above procedure, we get
\begin{eqnarray*}
[(r_{1}a_{i_{2}}b_{j_{2}}r_{2})(r_{1}a_{i_{1}}b_{j_{1}}r_{2})^{s+1})]^{2n+2} = 0.
\end{eqnarray*}
This implies that $(r_{1}a_{i_{2}}b_{j_{2}}r_{2})(r_{1}a_{i_{1}}b_{j_{1}}r_{2})^{s+1} \in N(I)$. Similarly, we can see that $(r_{1}a_{i_{t}}b_{j_{t}}r_{2})\\(r_{1}a_{i_{1}}b_{j_{1}}r_{2})^{s+1} \in N(I)$ for $3 \leq t \leq k$. Since $I$ is a semicommutative ideal, so $N(I)$ is an ideal of $I$, therefore $\sum_{t=2}^{k}(r_{1}a_{i_{2}}b_{j_{2}}r_{2})(r_{1}a_{i_{1}}b_{j_{1}}r_{2})^{s+1} \in N(I)$. On the other hand, multiplying by $(r_{1}a_{i_{1}}b_{j_{1}}r_{2})^{s+1}$ from right in equation $(2.2)$, we have $\sum_{t=1}^{k}(r_{1}a_{i_{t}}b_{j_{t}}r_{2})(r_{1}a_{i_{1}}b_{j_{1}}r_{2})^{s+1} \in N(I)$, so $(r_{1}a_{i_{1}}b_{j_{1}}r_{2})^{s+2} \in N(I)\subseteq N(R)$ and hence, $a_{i_{1}}b_{j_{1}} \in N_{*}(R)$. Continuing this process, we get $a_{i_{t}}b_{j_{t}} \in N_{*}(R)$ for $1\leq t \leq k$. So, $a_{i}b_{j} \in N_{*}(R)$ for each $i, j$ with $g_{i}h_{j} = w$. Thus, $a_{i}b_{j} \in N_{*}(R)$ for each $i, j$.
\end{proof}

Let $R$ be a ring with an endomorphism $\alpha$ such that $\alpha(1) = 1$. Chen et al. in \cite{JJJ}, considered the skew upper triangular matrix ring as a set of all upper triangular matrices with operations usual addition of matrices and multiplication subjected to the condition $E_{ij}r = \alpha^{j-i}(r)E_{ij}$, i.e. for any two matrices $(a_{ij})$ and$(b_{ij})$, we have $(a_{ij})(b_{ij}) = (c_{ij})$, where $c_{ij} = a_{ij}b_{ij}+a_{i, i+1}\alpha(b_{i+1, j})+\cdots+a_{ij}\alpha^{j-i}(b_{jj})$ for each $i\leq j$ and it is denoted by $T_{n}(R, \alpha)$.\\
It is noted that $N_{*}(T_{n}(R, \alpha)) = (N_{*}(R), R,\ldots, R)$\\
The subring of the skew triangular matrices with constant main diagonal is denoted by $S(R, n, \alpha)$. Also, the subring of skew triangular matrices with constant diagonals is denoted by $T(R, n, \alpha)$. We can denote $A = (a_{ij}) \in T(R, n, \alpha)$ by $(a_{11}, a_{12}, \ldots, a_{1n})$. Then $T(R, n, \alpha)$ is a ring with addition is pointwise and multiplication given by
\begin{eqnarray*}
(a_{0}, a_{1}, \ldots, a_{n-1})(b_{0}, b_{1}, \ldots, b_{n-1}) = (a_{0}b_{0}, a_{0}\ast b_{1}+a_{1}\ast b_{0}, \ldots, a_{0}\ast b_{n-1}+\cdots+a_{n-1}\ast b_{0})
\end{eqnarray*}
with $a_{i}\ast b_{j} = a_{i}\alpha^{i}(b_{j})$ for each $i, j$. On the other hand, there is a ring isomorphism $\Phi : R[x, \alpha]/ (x^{n})\rightarrow T(R, n, \alpha)$, given by $\Phi (\sum_{i=0}^{n-1}a_{i}x^{i}) = (a_{0}, a_{1}, a_{2}, \ldots, a_{n-1})$, with $a_{i} \in R,$ $0 \leq i \leq n-1$. So, $T(R, n, \alpha) \cong R[x, \alpha]/(x^{n})$, where $R[x, \alpha]$ is the skew polynomial ring with multiplication subject to the condition $xr = \alpha(r)x$ for each $r \in R$.\\
Also, we consider the following subrings of $S(R, n, \alpha)$
\begin{eqnarray*}
A(R, n, \alpha) = \left\{\sum_{j=1}^{\lfloor\frac{n}{2}\rfloor}\sum_{i=1}^{n-j+1}a_{j}E_{i, i+j-1}+ \sum_{j= \lfloor\frac{n}{2}\rfloor+1}^{n}\sum_{i=1}^{n-j+1}a_{i, i+j-1}E_{i+j-1} | a_{j}, a_{i, k} \in R \right\}.
\end{eqnarray*}

\begin{eqnarray*}
B(R, n, \alpha) = \{A+rE_{1k} ~|~ A \in A(R, n, \alpha), r \in R ~and~ n = 2k\geq 4\}.
\end{eqnarray*}
Let $S$ be a monoid and $f$ be an identity of $S$. Suppose $F\cup \{0\}$ is a free monoid generated by $W = \{w_{1}, w_{2}, \ldots, w_{t}\}$ and $S$ is a factor of $F$. Setting certain monomial in $W$ to $0$, it is enough to show that for some $n$, $\delta^{n} = 0$ for any $\delta \neq f$. Let $R$ be a ring with an endomorphism $\alpha$. Then we can form the skew monoid ring $R[S, \alpha]$, by taking its elements to be finite formal combinations $\sum_{g\in S}r_{g}g$, with multiplication subject to the relation $w_{i}r = \alpha(r)w_{i}$, for each $1 \leq i \leq t$. It is easily seen that $N_{*}(R[S, \alpha]) = \{\sum_{g \in S}r_{g}g | r_{g} \in N_{*}(R)\}$ and also the ring $S(R, n, \alpha)$ and $T(R, n, \alpha)$ fit naturally into $R[S, \alpha]$ with $W = \{E_{12}, \ldots, E_{n-1, n}\}$ and $W = \{E_{12}+\cdots+E_{n-1, n}\}$ respectively.
\begin{thm}
Let $M$ and $S$ be two monoids and $\alpha$ an endomorphism of the ring $R$ with $\alpha(1) = 1$. Then the following hold:
\begin{itemize}
\item[$(1)$] $R$ is lower nil $M$-Armendariz if and only if so is $S(R, n, \alpha)$
\item[$(2)$] $R$ is lower nil $M$-Armendariz if and only if so is $T(R, n, \alpha)$
\item[$(3)$] $R$ is lower nil $M$-Armendariz if and only if so is $A(R, n, \alpha)$
\item[$(4)$] $R$ is lower nil $M$-Armendariz if and only if so is $B(R, n, \alpha)$
\item[$(5)$] $R$ is lower nil $M$-Armendariz if and only if so is $T_{n}(R, \alpha)$
\item[$(6)$] $R$ is lower nil $M$-Armendariz if and only if so is $R(S, \alpha]$
\end{itemize}
\end{thm}
\begin{proof}
\begin{itemize}
\item[$(1)$] Suppose that $R$ is a lower nil-$M$ Armendariz ring. Let $\alpha = \sum_{i=1}^{m}A_{i}g_{i}$ and $\beta = \sum_{j=1}^{n}B_{j}h_{j} \in S(R, n, \alpha)[M]$ such that $\alpha\beta \in N_{*}(S(R, n, \alpha))[M]$, where $A_{i} = (a_{st}^{(i)}),$ $B_{j} = (b_{st}^{(j)})$ for $1 \leq i \leq m$ and $1 \leq j \leq n$. This implies $\alpha_{0}\beta_{0} \in N_{*}(R)[M]$, where $\alpha_{0} = \sum_{i=1}^{m}a_{11}^{(i)}g_{i}$ and $\beta_{0} = b_{11}^{(j)}h_{j} \in R[M]$. Therefore, $a_{11}^{(i)}b_{11}^{(j)} \in N_{*}(R)$ for each $i, j$, since $R$ is a lower nil $M$-Armendariz ring. Therefore, $A_{i}B_{j} \in N_{*}(S(R, n, \alpha))$ for each $i, j$. Hence, $S(R, n, \alpha)$ is a lower nil $M$-Armendariz ring. Proof of other part is similar.
\end{itemize}
\end{proof}

\begin{thm}(Theorem 3.2 of \cite{A}) Let $M$ be a u.p. monoid with nontrivial center. If $R$ is a lower nil $M$-Armendariz ring, then $N(R)[M] = N(R[M])$.
\end{thm}

\begin{thm}(Theorem 3 of \cite{J}) Let $R$ be a ring and $M$ a u.p. monoid. Then $N_{*}(R)[M] = N_{*}(R[M])$.
\end{thm}
\begin{thm} (Theorem 5 of \cite{J}) Let $R$ be a ring and $M$ a u.p. monoid. Then $N_{*}(R) = N(R)$ if and only if $N_{*}(R[M]) = N(R[M])$.
\end{thm}
\begin{rem}\label{rem1} For a $2$-primal ring $R$ and a u.p. monoid $M$, one can conclude that $N(R)[M] = N(R[M])$.
\end{rem}

\textbf{Weak Annihilator:} Let $R$ be a ring. For a subset $X$ of $R$, the weak annihilator of $X$ in $R$ is $N_{R}(X) = \{a \in R : xa \in N(R) ~for ~all ~x \in X\}$.
If $X$ is singleton, say $X = \{r\}$, then we use $N_{R}(r)$.\\
Given a ring $R$, we define

\begin{equation*}
N Ann_{R}(2^{R}) = \{N_{R}(U): U \subseteq R \}
\end{equation*}
\begin{equation*}
N Ann_{R[M]}(2^{R[M]}) = \{N_{R[M]}(V): V \subseteq R[M] \}.
\end{equation*}
For an element $\beta \in R[M],$ $C_{\beta}$ denotes the set of all coefficients of $\beta$ and for a subset $W$ of $R[M]$, $C_{W}$ denotes the set $\bigcup_{\beta \in W}C_{\beta}$.
\begin{thm} Let $M$ be a u. p. monoid with nontrivial center and $R$ a lower nil $M$-Armendariz. Then
\begin{equation*}
\Phi : N Ann_{R}(2^{R})\rightarrow N Ann_{R[M]}(2^{R[M]})
\end{equation*}
defined by $\Phi(I) = I[M]$ for every $I \in N Ann_{R}(2^{R})$ is bijective.
\end{thm}
\begin{proof} By Theorem (3.2) of \cite{A}, we have $N(R)[M] = N(R[M])$ and also from Proposition \ref{pro6}, $R$ is nil $M$-Armendariz ring. Therefore by Theorem $(3.1)$ of \cite{O} result is true.
\end{proof}
A ring $R$ is said to be nilpotent p.p. ring if for any $q \notin N(R)$, the $N_{R}(q)$ is generated as a right ideal by a nilpotent element of $R$.
\begin{thm} Let $M$ be a monoid and $R$ be a semicommutative lower nil $M$-Armendariz ring. If $R$ is nilpotent p.p. ring, then so is $R[M]$.
\end{thm}
\begin{proof} Let $\alpha = a_{1}g_{1}+a_{2}g_{2}+\cdots+a_{n}g_{n} \notin N(R[M])$ and $\beta = b_{1}h_{1}+b_{2}h_{2}+\cdots+b_{m}h_{m} \in N_{R[M]}(\alpha)$. Then $\alpha \beta \in N(R[M]$. Also, by Proposition (2.12) of \cite{AA}, $N(R)[M] = N(R[M])$, therefore $\alpha \beta \in N(R)[M]$ and hence $a_{i}b_{j} \in N(R)$ for each $1 \leq i \leq n, 1 \leq j \leq m$. Since $\alpha \notin N(R[M]) = N(R)[M]$, so there exist at least one $i$, $1\leq i \leq n$ such that $a_{i} \notin N(R)$. Therefore, $R$ being nilpotent $p. p.$ ring, there exist some $s \in N(R)$ such that $N_{R}(a_{i}) = sR$. \\
Now, we claim $N_{R[M]}(\alpha) = se.R[M]$. Since $b_{j} \in N_{R}(a_{i})$ for each $1 \leq j \leq m$, so $b_{j} = sr_{j}$ for some $r_{j} \in R$. Therefore $\beta = se(r_{1}h_{1}+r_{2}h_{2}+\cdots+r_{m}h_{m}) \in se.R[M]$, hence $N_{R[M]}(\alpha) \subseteq se.R[M]$. \\
Also, for any $\delta = w_{1}e_{1}+w_{2}e_{2}+\cdots+w_{q}e_{q} \in R[M]$ and since $s \in N(R)$, therefore $a_{i}sw_{j} \in N(R)$ for each $i, j$, because $R$ is semicommutative ring. So $\alpha.se. \delta \in N(R)[M] = N(R[M])$. Hence we have $se.R[M] \subseteq N_{R[M]}(\alpha)$. Thus, $N_{R[M]}(\alpha) = se.R[M]$, where $se \in N(R[M])$.
\end{proof}

\section{Lower nil $M$-Armendarz ring and $M$-Armendariz ring}
\begin{thm}\label{thm1} Let $M$ be a strictly totally ordered monoid and $I$ be a proper ideal of $R$ with $N_{*}(R) \subseteq I.$ If $R/I$ is $M$-Armendariz and $I$ is 2-primal ring, then $R$ is lower nil $M$-Armendariz ring.
\end{thm}
\begin{proof} Suppose $\overline{R} = R/I$ is an $M$-Armendariz ring and $I$ is a 2-primal ring. Since, $N_{*}(R) \subseteq I$, so $N_{*}(R) \subseteq N_{*}(I)$. Also $I$ is an ideal of $R$, so $N_{*}(I) \subseteq N_{*}(R)$ and hence $N_{*}(I) = N_{*}(R)$. Since $I$ is a 2-primal, $ I/N_{*}(R) = I/N_{*}(I) \cong I/N(I)$, therefore it is reduced. \\
Let $\alpha = a_{1}g_{1}+a_{2}g_{2}+\ldots+a_{n}g_{n}$ and $\beta = b_{1}h_{1}+b_{2}h_{2}+\ldots+b_{m}h_{m}$ such that $\alpha \beta \in N_{*}(R)[M]$ with $g_{1}<g_{2}<\ldots<g_{m}, h_{1}<h_{2}<\ldots<h_{n}$ where $g_{1}, g_{2}, \ldots, g_{n}, h_{1}, h_{2}, \ldots h_{m} \in M$. Now, we use transfinite induction on strictly totally ordered set $(M, \leq)$ to show $a_{i}b_{j} \in N_{*}(R)$. Since $N_{*}(R)\subseteq I$, therefore $\alpha \beta \in I[M]$. Since $R/I$ is $M$-Armendariz, therefore $a_{i}b_{j} \in I$ for each $i$ and $j$. If there exist $1 \leq i \leq n$ and $1 \leq j \leq m$ such that $g_{i}h_{i} = g_{1}h_{1}$, then $g_{1} \leq g_{i}$ and $h_{1} \leq h_{j}$. If $g_{1}<g_{i}$, then $g_{1}h_{1}<g_{i}h_{1}\leq g_{i}h_{i} = g_{1}h_{1}$ implies $g_{1}h_{1}<g_{1}h_{1}$, a contradiction. Hence $a_{1}b_{1} \in N_{*}(R)$. \\
Now, let $w \in M$ be such that for any $g_{i}$ and $h_{j}$, $g_{i}h_{j}<w$, $a_{i}b_{j} \in N_{*}(R)$. We will show that $a_{i}b_{j} \in N_{*}(R)$ for any $g_{i}$ and $h_{j}$ with $g_{i}h_{j} = w$. Consider $X = \{(g_{i}, h_{j}) : g_{i}h_{j} = w\}$, then $X$ is a finite set. So, we put $X = \{(g_{i_t}, h_{j_t}): t = 1, 2, \ldots k\}$ such that $g_{i_1}<g_{i_2}<\ldots<g_{i_k}$. Here $M$ is cancellative, $g_{i_1} = g_{i_2}$ and $g_{i_1}h_{j_1} = g_{i_2}h_{j_2} = w$ imply $h_{j_1} = h_{j_2}$. Since $(M, \leq)$ strictly totally ordered monoid, $g_{i_1}<g_{i_2}$ and $g_{i_1}h_{j_1} = g_{i_2}h_{j_2} = w$, therefore $h_{j_2}<h_{j_1}$. Thus, we have $h_{j_k}<\ldots<h_{j_2}<h_{j_1}$. Now \\
\begin{equation}
\sum_{g_{i}h_{j} = w}a_{i}b_{j} = \sum_{t=1}^{t=k}a_{i_t}b_{j_t} \in N_{*}(R).
\end{equation}
For any $t\geq 2$, $g_{i_1}h_{j_t} < g_{i_t}h_{j_t} = w$ and by induction hypothesis $a_{i_1}b_{j_t}\in N_{*}(R)$. Take $p = b_{j_t}(a_{i_1}b_{j_1})a_{i_1}$, then $p \in I$, since $a_{i}b_{j} \in I$ for each $i$ and $j$. Since $a_{i_1}b_{j_t}\in N_{*}(R)$, so $p^{2} = b_{j_t}(a_{i_1}b_{j_1})a_{i_1}b_{j_t}(a_{i_1}b_{j_1})a_{i_1} \in N_{*}(R)$. Since $I/N_{*}(R)$ is reduced, therefore, $p \in N_{*}(R)$. Now,\\
\begin{equation}
(a_{i_t} b_{j_t})(a_{i_1}b_{i_1})^{2} = (a_{i_t} (b_{j_t}(a_{i_1} b_{i_1})a_{i_1})b_{i_1})\in N_{*}(R)
\end{equation}
Multiplying $(2.1)$ by $(a_{i_1}b_{i_1})^{2}$ from right, we get\\
$(a_{i_1}b_{i_1})^{3} + \sum_{t=2}^{k}(a_{i_t} b_{i_t})(a_{i_1})b_{i_1})^{2} \in N_{*}(R)$. This implies $(a_{i_1}b_{i_1})^{3} \in N_{*}(R)$. Since $I/N_{*}(R)$ is reduced, therefore, $(a_{i_1}b_{i_1}) \in N_{*}(R)$. \\
Again by $(2.1)$ we have \\
\begin{equation}
a_{i_2}b_{i_2}+\ldots+a_{i_k}b_{i_k} \in N_{*}(R)
\end{equation}
so, $(a_{i_t} b_{i_t})(a_{i_2} b_{i_2})^{2} = (a_{i_t} b_{i_t})(a_{i_2} b_{i_2})(a_{i_2} b_{i_2}) \in N_{*}(R)$. Multiplying $(2.3)$ by $(a_{i_2}b_{i_2})^{2}$ from right side, we get $(a_{i_2}b_{i_2})^{3}\in N_{*}(R)$. This implies $(a_{i_2}b_{i_2})\in N_{*}(R)$, because $I/N_{*}(R)$ is reduced. Continuing the procedure, we get $a_{i_t}b_{j_t} \in N_{*}(R)$ for $t = 1, 2, \ldots k$ with $g_{i}h_{j} = w$. Therefore, by transfinite induction, $a_{i}b_{j} \in N_{*}(R)$, for each $i$ and $j$. Thus, $R$ is a lower nil $M$-Armendariz ring.
\end{proof}
Recall that a monoid $M$ is said to be torsion-free if the following property hold if $g, h \in M$ and $k\geq 1$ are such that $g^{k} = h^{k}$ implies $g = h$.
\begin{cor} Let $M$ be a commutative, cancellative and torsion free monoid with $|M|\geq 2$. If either one of the following conditions holds, then $R$ is a lower nil $M$-Armendariz ring.
\begin{itemize}
\item[$(1)$] $R$ is a 2-primal ring.
\item[$(2)$] $R/I$ is an $M$-Armendariz ring for some ideal $I$ of $R$ and $I$ is $2$-primal with $N_{*}(R)\subseteq I$.
\end{itemize}
\end{cor}
\begin{proof}
If $M$ is commutative, cancellative and torsion-free monoid, then by Result $(3.3)$ of \cite{P}, there exist a compatible strict total order monoid $\leq$ on $M$. Then by Proposition \ref{pro1} and Theorem \ref{thm1} results hold.
\end{proof}
\begin{pro} (Proposition $(3.1)(a)$ of \cite{A})\label{pro5}For any u.p. monoid $M$, every $M$-Armendariz ring are lower nil $M$-Armendariz ring.
\end{pro}
A ring $R$ is a right (resp. left) uniserial ring if its lattice of right (resp. left) ideals is totally ordered by inclusion. Right uniserial rings are also called right chain ring or right valuation rings because they are obvious generalization of commutative valuation domains. Like commutative valuation domains, right uniserial rings have a rich theory and they offer remarkable examples, we refer \cite{GG}.
\begin{pro} For any u.p. monoid $M$, every right or left uniserial ring is a lower nil $M$-Armendariz ring.
\end{pro}
\begin{proof} Let $R$ be a uniserial ring and $M$ a u.p. monoid. By Corollary $(6.2)$ of \cite{GG}, $R$ is $M$-Armendariz ring. Therefore, by Proposition \ref{pro5}, $R$ is a lower nil $M$-Armendariz ring.
\end{proof}

\begin{pro}
Let $M$ be a monoid and $N$ a u.p.-monoid. If $R$ is semicommutative ring as well as $M$-Armendariz ring, then $R[M]$ is a lower nil $N$-Armendariz ring.
\end{pro}

\begin{proof} Let $\alpha = a_{1}g_{1}+a_{2}g_{2}+\ldots+a_{n}g_{n}$ and $\beta = b_{1}h_{1}+b_{2}h_{2}+\ldots+b_{m}h_{m}\in R[M]$ such that $\alpha \beta = 0$. This implies $a_{i}b_{j} = 0$ for each $i$ and $j$, since $R$ is $M$-Armendariz ring. Also $R$ is semicommutative, therefore $a_{i}Rb_{j} = 0$ for each $i, j$. Again, we can easily see that $\alpha R[M]\beta = 0$. So from Corollary \ref{cor1}, $R[M]$ is a lower nil $N$-Armendariz ring.
\end{proof}

\begin{lem}\label{lem1} Let $M$ be a monoid. If $R$ is a $2$-primal $M$-Armendariz ring, then $R[M]$ is $2$-primal ring and $R$ is lower nil $M$-Armendariz ring with $N(R)[M] = N(R[M])= N_{*}(R)[M] = N_{*}(R[M])$.
\end{lem}
\begin{proof}
By Theorem $(2.3)$ of \cite{O}, we have $R$ is a nil $M$-Armendariz ring. Since $R$ is $2$-primal, therefore $R$ is lower nil $M$-Armendariz ring. Also by Lemma $(2.2)$ and Theorem $(2.3)$ of \cite{O}, we have $N(R)[M] = N(R[M])$ and $R$ is $2$-primal, so $R[M]$ is $2$-primal. Hence, $N(R)[M] = N(R[M]) = N_{*}(R)[M] = N_{*}(R[M])$.
\end{proof}
\begin{pro}
 Let $M$ be a monoid and $N$ a u.p.-monoid. If $R$ is a $2$-primal $M$-Armendariz ring, then $R[N]$ is a lower nil $M$-Armendariz ring.
\end{pro}
\begin{proof} Since $N$ is a u.p.-monoid, therefore $N_{*}(R)[N] = N_{*}(R[N])$, by [\cite{J}, Theorem 3]. Also, from Lemma \ref{lem1} $N_{*}(R)[M] = N_{*}(R[M])$. Next, from Lemma \ref{lem1}, $R[M]$ is a lower nil $N$-Armendariz ring. Also, rings $R[N][M]$ and $R[M][N]$ are isomorphic under the following map\\
\begin{equation*}
  \sum_{p}\Big(\sum_{i}a_{ip}n_{i}\Big)m_{p}\longmapsto \sum_{i}\Big(\sum_{p}a_{ip}m_{p}\Big)n_{i}.
\end{equation*}
Now, let\\
\begin{equation*}
\Big(\sum_{i}\alpha_{i}g_{i}\Big)\Big(\sum_{j}\beta_{j}h_{j}\Big)\in N_{*}(R[N])[M],
\end{equation*}
where $\alpha_{i}, \beta_{j} \in R[N]$ and $g_{i}, h_{j}\in M$. We claim that $\alpha_{i}\beta_{j} \in N_{*}(R[N])$ for each $i$ and $j$. Let\\
\begin{equation*}
\alpha_{i} = \sum_{p}a_{ip}n_{p},~ \beta_{j} = \sum_{q}b_{jq}n^{'}_{q} \in R[N].
\end{equation*}
Then\\
\begin{equation*}
\Big(\sum_{i}\Big(\sum_{p}a_{ip}n_{p}\Big)g_{i}\Big)\Big(\sum_{j}\Big(\sum_{q}b_{jq}n^{'}_{q}\Big)h_{j}\Big) \in N_{*}(R[N])[M] = N_{*}(R)[N][M].
\end{equation*}
Thus, we have\\
\begin{equation*}
\Big(\sum_{p}\Big(\sum_{i}a_{ip}g_{i}\Big)n_{p}\Big)\Big(\sum_{q}\Big(\sum_{j}b_{jq}h_{j}\Big)n^{'}_{q}\Big) \in N_{*}(R)[M][N] = N_{*}(R[M])[N].
\end{equation*}
By Lemma \ref{lem1}, $R[M]$ is a lower nil $N$-Armendariz ring, so\\
\begin{equation*}
\Big(\sum_{i}a_{ip}g_{i}\Big)\Big(\sum_{j}b_{jq}h_{j}\Big) \in N_{*}(R[M]) = N_{*}(R)[M]
\end{equation*}
for all $p, q$. Hence, $a_{ip}b_{jq} \in N_{*}(R)$  for each $i, j, p, q$ and from Lemma \ref{lem1}, $R$ is a lower nil $M$-Armendariz ring. So\\
\begin{equation*}
\alpha_{i}\beta_{j} = \Big(\sum_{p}a_{ip}n_{i}\Big)\Big(\sum_{q}b_{jq}n^{'}_{j}\Big) \in N_{*}(R)[N] = N_{*}(R[N])
\end{equation*}
for each $i, j$. Thus, $R[N]$ is a lower nil $M$-Armendariz ring.
\end{proof}
\begin{pro} Let $M$ be a monoid and $N$ a u.p.-monoid. If $R$ is $2$-primal $M$-Armendariz ring, then $R$ is a lower nil $M\times N$-Armendariz ring.
\end{pro}
\begin{proof} Construction of this proof is based on the proof of Theorem $[2.3]$ of \cite{Z}.\\ Let  $\sum_{i=1}^{t}a_{i}(m_{i}, n_{i})\in R[M\times N]$. Without loss of generality, we assume that $\{n_{1}, n_{2}, \ldots, n_{t}\} = \{n_{1}, n_{2}, \ldots, n_{s}\}$ with $n_{i} \neq n_{j}$, where $1 \leq i \neq j \leq s$.
\begin{equation*}
Let ~~~~~~~~~~~A_{p} = \{i : 1 \leq i \leq t, n_{i} = n_{p}\}, for ~1 \leq p \leq s.
\end{equation*}
Then
\begin{equation*}
\sum_{i=1}^{s}\Big(\sum_{i \in A_{p}}a_{i}m_{i}\Big)n_{p} \in R[M][N].
\end{equation*}
Note that $m_{i} \neq m_{i^{'}}$ for any $i, i^{'} \in A_{p}$, when $i \neq i^{'}$. It is easy to see that there exists an isomorphism between rings $R[M\times N]$ and $R[M][N]$ defined by\\
\begin{equation*}
\sum_{i=1}^{t}a_{i}(m_{i}, n_{i})\mapsto \sum_{p=1}^{s}\Big(\sum_{i \in A_{p}}a_{i}m_{i}\Big)n_{p}.
\end{equation*}
Assume
\begin{equation*}
\Big(\sum_{i=1}^{t}a_{i}(m_{i}, n_{i})\Big)\Big(\sum_{j=1}^{t^{'}}b_{j}(m^{'}_{j}, n^{'}_{j})\Big) \in N_{*}(R[M\times N]).
\end{equation*}
Then from the above isomorphism, it follows that
\begin{equation*}
\Big(\sum_{p=1}^{s}\Big(\sum_{i \in A_{p}}a_{i}m_{i}\Big)n_{p}\Big)\Big(\sum_{q=1}^{s^{'}}\Big(\sum_{j \in B_{q}}b_{j}m^{'}_{j}\Big)n^{'}_{q}\Big) \in N_{*}(R[M][N]).
\end{equation*}
By Lemma \ref{lem1}, $R[M]$ is a lower nil $N$-Armendariz ring. Therefore,
\begin{equation*}
\Big(\sum_{i \in A_{p}}a_{i}m_{i}\Big)\Big(\sum_{j \in B_{q}}b_{j}m^{'}_{j}\Big) \in N_{*}(R[M]) = N_{*}(R)[M].
\end{equation*}
Also, by Lemma \ref{lem1}, $R$ is a lower nil $M$-Armendariz ring. Therefore, $a_{i}b_{j} \in N_{*}(R)$ for each $i \in A_{p}$ and $j \in B_{q}$. Moreover, $a_{i}b_{j} \in N_{*}(R)$ for each $i$ and $j$, where $1 \leq i \leq t$ and $1 \leq j \leq t^{'}$. Hence, $R$ is a lower nil $M\times N$-Armendariz ring.
\end{proof}




\begin{thebibliography}{00}
\bibitem {A} A. Alhevaz, E. Hashemi and M. Ziembowski, Nilradicals of unique product monoid rings, J. Algebra Appl. 16(5) (2016), 1750133-18.
\bibitem {AA} A. Alhevaz and A. Moussavi, On monoid rings over nil Armendariz ring, Comm. Algebra 42(1) (2014), 1-21.

\bibitem {E} E. P. Armendariz, A note on extensions of Baer and P.P.-rings, J. Austral. Math.Soc. 18 (1974), 470-473.

\bibitem {G} G. F. Birkenmeier, H. E. Heatherly and E. K. Lee, Completely prime ideals and associated radicals, in: S. K. Jain, S. T. Rizvi (Eds.), Proc. Biennial Ohio State Denison Conference (1992), World Scientific, New Jersey (1993), 102-129.

\bibitem {JJJ} J. Chen, X. Yang, Y. Zhou, On strongly clean matrix and triangular matrix rings, Comm. Algebra 34(10) (2006), 3659-3674.

\bibitem {J} J. S. Cheon and J. A. Kim, Prime radicals in u.p.-monoid rings, Bull. Korean. Math. Soc. 49(3) (2012), 511-515.

\bibitem {K} K. R. Goodearl, Von Neumann Regular Rings, Pitman, London, 1979.

\bibitem {EE} E. Hashemi, Nil-Armendariz rings relative to a monoid, Mediterr. J. Math. 10(1) (2013), 111-121.

\bibitem {S} S.U. Hwang, Y.C. Jeon and Y. Lee, Structure and topological conditions of NI rings, J. Algebra 302 (2006), 186-199.

\bibitem {YYY} Y. C. Jeon, N. K. Kim and Y. Lee and J. S. Yoon, On weak Armendariz ring, Bull. Korean Math. Soc. 46(1) (2009), 135-146.


\bibitem {Z} Z. K. Liu, Armendariz rings relative to a monoid, Comm. Algebra 33(3) (2005), 649-661.

\bibitem {O} O. Lunqun and  L. Jinwang, Nil-Armendariz rings relative to a monoid, Arab. J. Math. 2(1) (2013), 81-90.

\bibitem {GGG} G. Marks, On 2-primal Ore extensions, Comm. Algebra 29(5) (2001), 2113-2123.

\bibitem {GG} G. Marks, R. Mazurek, and  M. Ziembowski A unified approach to various generalizations of Armendariz rings. Bull. Aust. Math. Soc. 81 (2010), 361-397.
\bibitem {L} L. Motais de Narbonne, Anneaux semi-commutatifs et unisériels anneaux dont les idéaux principaux sont idempotents, in:
Proceedings of the 106th National Congress of Learned Societies, Perpignan, 1981, Bib. Nat., Paris, (1982), 71-73.


\bibitem {JJ} J. Okninski, Semigroup Algebras, Monographs and Textbooks in Pure and Applied Mathematics, Vol. 138, Marcel Dekker, New York, (1991).

\bibitem {D} D. S. Passman, The Algebraic Structure of Group Rings Wiley, New York, (1977).

\bibitem {P} P. Ribenboim, Noetherian rings of generalized power series. J. Pure Appl. Algebra 79 (1992), 293-312.

\bibitem {MM} M. B. Rege and S. Chhawchharia, Armendariz rings, Proc. Japan Acad. Ser. A Math. Sci. 73 (1997), 14-17.
\end{thebibliography}
\end{document}